\documentclass[11pt,times,twoside]{article}
\usepackage[body={14.5cm,22cm}, left=2.5cm, right=2.5cm, top=3cm]{geometry}
\usepackage{graphicx,amsmath,amsthm,latexsym,amssymb,amstext}
\usepackage[pdfstartview=FitH,
bookmarksnumbered=true,bookmarksopen=true,%
colorlinks=true,pdfborder=001,citecolor=blue,%
linkcolor=blue,urlcolor=blue]{hyperref}

\newtheorem{theorem}{Theorem}[section]
\newtheorem{definition}[theorem]{Definition}

\newtheorem{proposition}[theorem]{Proposition}

\newtheorem{remark}[theorem]{Remark}
\newtheorem{example}{Example}
\numberwithin{equation}{section}

\begin{document}

\title{\Large \bf Functional It\^{o} formula for fractional Brownian motion}

\author
{\textbf{Jiaqiang Wen}$^{1}$,
\textbf{Yufeng Shi}$^{1,2,}$\thanks{Corresponding author. Email: jqwen@mail.sdu.edu.cn (J. Wen), yfshi@sdu.edu.cn (Y. Shi)} \vspace{3mm} \\
\normalsize{$^{1}$Institute for Financial Studies and School of Mathematics,}\\
\normalsize{Shandong University, Jinan 250100, China}\\
\normalsize{$^{2}$School of Statistics, Shandong University of Finance and Economics,}\\
\normalsize{Jinan 250014, China}\\
}

\date{}

\renewcommand{\thefootnote}{\fnsymbol{footnote}}

\footnotetext[0]{This work is supported by
    National Natural Science Foundation of China (Grant Nos. 11371226 and 11231005),
    Foundation for Innovative Research Groups of National Natural Science Foundation of China (Grant No. 11221061),
    the 111 Project (Grant No. B12023).}

\maketitle

\begin{abstract}
We develop the functional It\^{o}/path-dependent calculus with respect to fractional Brownian motion with Hurst parameter $H> \frac{1}{2}$.
Firstly, two types of integrals are studied.
The first type is Stratonovich integral, and the second type is Wick-It\^{o}-Skorohod integral.
Then we establish the functional It\^{o} formulas for fractional Brownian motion,
which extend the functional It\^{o} formulas in Dupire (2009) and Cont and Fourni\'{e} (2013) to the case of non-semimartingale.
Finally, as an application, we deal with a class of fractional backward stochastic differential equations (BSDEs).
A relation between fractional BSDEs and path-dependent partial differential equations (PDEs) is established.
\end{abstract}

\textbf{Keywords}:
Functional It\^{o} calculus;
Functional It\^{o} formula;
Fractional Brownian motion;
Fractional BSDE;
Path-dependent PDE

\vspace{3mm}

\textbf{2010 Mathematics Subject Classification}: 60H05; 60H07; 60G22

\section{Introduction}

Recently, a new branch of stochastic calculus has appeared, known as functional It\^{o} calculus,
which is an extension of classical It\^{o} calculus to functionals depending on all pathes of a stochastic process and not only on its current values,
see Dupire \cite{Dup} for the initial point of view and Cont and Fourni\'{e} \cite{Cont10,Con} for further developed results.
A new type of functional It\^{o} formulas for semimartingale was also established in \cite{Dup,Con}.
Since then, both the theory and application of the functional It\^{o} calculus have been paid very strong attention.
We refer to  Buckdahn, Ma and Zhang \cite{Buckdahn}, Cosso and Russo \cite{Russo4}, Ekren et al. \cite{Iekr},
Ekren, Touzi and Zhang \cite{Iekr2,Iekr3}, Keller and Zhang \cite{Keller}, Peng and Wang \cite{Spen}, and  Tang and Zhang \cite{Tang} etc.,
for recent developments on functional It\^{o} calculus for semimartingale.

As we know that It\^{o} formula is an important ingredient and a powerful tool in It\^{o} calculus,
it is a natural and curious question if there is a functional It\^{o} formula for non-semimartingale.
It is well-known that fractional Brownian motion (fBm, for short) is not a semimartingale and plays an increasingly important role in many fields such as hydrology, economics and telecommunications.
Hence, it is a significant and challenging problem to develop the functional It\^{o} calculus with respect to fractional Brownian motion,
especially to establish the functional It\^{o} formula for fractional Brownian motion.
The functional It\^{o} formula with respect to a process with finite quadratic variation was derived in
Cont and Fourni\^{e} \cite{Cont10} by using the discretization techniques of F\"{o}llmer \cite{Follmer} type,
where the integral is F\"{o}llmer integral and the integrand is non-anticipative.
Cosso and Russo \cite{Russo4} also obtained a functional It\^{o} formula with respect to a process with finite quadratic variation via regularization approach,
where they considered the forward integral (Russo-Vollois integral \cite{Russo}).

Fractional Brownian motion with Hurst parameter $H \in (0, 1)$ is a zero mean Gaussian process $B^{H}=\{ B^{H}(t),t\geq 0 \}$ whose covariance is given by
\begin{equation}\label{1}
  \mathbb{E}\big[B^{H}(t)B^{H}(s)\big] = \frac{1}{2} (t^{2H} + s^{2H} - |t-s|^{2H}).
\end{equation}
If $H=\frac{1}{2}$, then the corresponding fractional Brownian motion is a classical Brownian motion.
When $H > \frac{1}{2}$, the process $B^{H}$ exhibits a long range dependence.
These properties make the fractional Brownian motion to be a useful driving noise in models arising in  finance,
 physics, telecommunication networks and other fields.
However, since the fractional Brownian motion is not a semimartingale,
the beautiful classical theory of stochastic calculus can not be applied to fractional Brownian motion.
As we know, there are essentially several different approaches in the literature in order to define stochastic integrals with respect to the fBm.
For example, Lin \cite{Lin} and Dai and Heyde \cite{Dai} introduced a stochastic integral as the limit of Riemann sums in the case $H > \frac{1}{2}$,
and an It\^{o} formula for fractional Brownian motion was obtained in \cite{Dai}.
Z\"{a}hle \cite{Zaehle} introduced a pathwise stochastic integral with respect to the fBm with parameter $H\in(0, 1)$.
The techniques of Malliavin calculus have firstly been used to develop the stochastic calculus for the fBm
in the pioneering work of Decreusefond and \"{U}st\"{u}nel \cite{Decreusefond}.
Along this way, this idea has been developed by many authors including theories and applications.
We refer to the works of Al\`{o}s, Mazet and Nualart \cite{Alos2}, Carmona and Coutin \cite{Carmona},
Duncan, Hu and Pasik-Duncan \cite{Hu}, Hu, Jolis and Tindel \cite{Huy}, Hu, Nualart and Song \cite{Songx},
Hu and {\O}ksendal \cite{Hu4}, Hu and Peng \cite{Peng}, and Nualart \cite{Nualart2} among others.
The advantage of the integral constructed by this method is that it has zero mean, and can be obtained as the limit of Riemann sums defined using Wick products.

In this paper, the functional It\^{o} calculus with respect to fractional Brownian motion with Hurst parameter $H> \frac{1}{2}$ is developed.
We firstly study two types of functional integrals.
The first type of integrals is Stratonovich integral with respect to c\`{a}dl\`{a}g (right continuous with left limits) process.
In particular, we discuss this kind of integral with respect to classical Brownian motion and fractional Brownian motion respectively.
The second type is Wick-It\^{o}-Skorohod integral with respect to fractional Brownian motion.
We emphasize that the integrands are allowed to be anticipative.
The mutual relations between these types of integrals  are studied.
It should be noted that there is a little difference in the relation between It\^{o} integral and Stratonovich integral
comparing with the classical case that they are both driven by classical Brownian motion.
Then we establish the functional It\^{o} formulas with respect to fractional Brownian motion,
where the integrals are of Stratonovich type and Wick-It\^{o}-Skorohod type respectively,
which extend the functional It\^{o} formulas for semimartingale (see Dupire \cite{Dup} and Cont and Fourni\'{e} \cite{Con}) to the case of non-semimartingale.

We point out that the proof to Theorem 3 in Cont and Fourni\'{e} \cite{Cont10} is nonprobabilistic,
while the approach to the first functional It\^{o} formula (Theorem \ref{32} below) is probabilistic.
The integral in the second functional It\^{o} formula (Theorem \ref{50} below) is Wick-It\^{o}-Skorohod integral,
which is different from the F\"{o}llmer integral in Cont and Fourni\'{e} \cite{Cont10}.
Moreover, the approach to obtain Theorem \ref{50} is Malliavin calculus approach.
In fact, our functional It\^{o} formulas are also different to the ones obtained in Cosso and Russo \cite{Russo4}
since the type of integrals and the approaches used in this paper are also different from Cosso and Russo \cite{Russo4}.
Finally, as an application, we deal with a class of fractional backward stochastic differential equations (BSDEs, for short).
A relation between fractional BSDEs with path-dependent coefficients and path-dependent partial differential equations (PDEs, for short) is also established.

This paper is organized as follows.
In Section 2, some existing results about functional It\^{o} calculus and fractional Brownian motion are presented.
Section 3 is devoted to studying Stratonovich type integral with respect to classical Brownian motion and fractional Brownian motion respectively.
A functional It\^{o} formula for fractional Brownian motion is also established in this section.
Wick-It\^{o}-Skorohod type integral with respect to the fractional Brownian motion is discussed in Section 4.
In Section 5, we deal with the fractional BSDEs with path-dependent coefficients.

\section{Preliminaries}

In this section, we recall some basic notions and results about functional It\^{o} calculus and fractional Brownian motion theory,
which are needed in the sequels.
The readers may refer to the articles such as Dupire \cite{Dup}, Cont and Fourni\'{e} \cite{Con},
Duncan et al. \cite{Hu}, Hu and {\O}ksendal \cite{Hu4} and Hu and Peng \cite{Peng}  for more details.

\subsection{Functional It\^{o} calculus}

 Let $T>0$ be fixed.
 For each $t\in [0,T]$, we denote $\Lambda_{t}$ the set of bounded c\`{a}dl\`{a}g $\mathbb{R}$-valued functions on $[0,t]$
 and $\Lambda=\bigcup_{t\in[0,T]}\Lambda_{t}$.
 For each $\gamma\in \Lambda_{T}$, the value of $\gamma$ at time $t$ is denoted by $\gamma(t)$ and the path of $\gamma$ up to time $t$ is
 denoted by $\gamma_{t}$, i.e., $\gamma_{t}=\gamma(r)_{0\leq r\leq t} \in \Lambda_{t}$.
 It is easy to see that $\gamma(r)=\gamma_{t}(r)$, for  $r\in [0,t]$.
 For each $\gamma_{t}\in \Lambda, \ s\geq t$ and $h\in \mathbb{R}$, we denote
  \[\begin{split}
    \gamma_{t}^{h}(r):=&\ \gamma(r)\mathbf{1}_{[0,t)}(r)+(\gamma(t)+h)\mathbf{1}_{\{t\}}, \ \ r\in [0,t],\\
      \gamma_{t,s}(r):=&\ \gamma(r)\mathbf{1}_{[0,t)}(r)+\gamma(t)\mathbf{1}_{[t,s]}(r), \ \ r\in[0,s].
  \end{split}\]
 It is clear that $\gamma_{t}^{h}\in \Lambda_{t}$ and $\gamma_{t,s}\in \Lambda_{s}$.
 For each $0\leq t\leq s\leq T$ and $\gamma_{t},\overline{\gamma}_{s}\in \Lambda$, we denote
\begin{equation*}
\begin{split}
       \|\gamma_{t}\|:=& \sup_{r\in[0,t]}|\gamma_{t}(r)|, \\
       \|\gamma_{t}-\overline{\gamma}_{s}\|:=&\sup_{r\in[0,s]}|\gamma_{t,s}(r)-\overline{\gamma}_{s}(r)|, \\
       d_{\infty}(\gamma_{t},\overline{\gamma}_{s}):=&\ \|\gamma_{t}-\overline{\gamma}_{s}\|+|t-s|.
\end{split}
\end{equation*}
 It is obvious that $\Lambda_{t}$ is a Banach space with respect to $\|\cdot \|$.
 Since $\Lambda$ is not a linear space, $d_{\infty}$ is not a norm.

 Now consider a function $F$ on $(\Lambda,d_{\infty})$.
 This function $F=F(\gamma_{t}^{h})$ can be regarded as a family of real valued functions:
 \begin{equation*}
   F(\gamma_{t}^{h})= F(t,\gamma(s)_{0\leq s<t},\gamma(t)+h), \ \gamma_{t}\in \Lambda_{t}, \ t\in[0,T], \ h\in \mathbb{R}.
 \end{equation*}
 It is also important to understand $F(\gamma_{t}^{h})$ as a function of $t, \ \gamma(s)_{0\leq s<t}, \ \gamma(t)$ and $h$.
 \begin{definition}
 A functional $F$ defined on $\Lambda$ is said to be continuous at $X_{t}\in \Lambda$, if for any $\varepsilon>0$ there exists $\delta>0$ such that
 for each $Y_{s}\in \Lambda$ with $d_{\infty}(X_{t},Y_{s})<\delta$,
 we have $|F(X_{t})-F(Y_{s})|<\varepsilon$.
 $F$ is said to be $\Lambda$-continuous if it is continuous at each $X_{t}\in \Lambda$.
 \end{definition}

 \begin{definition}
 Let $F:\Lambda\rightarrow \mathbb{R}$ and $X_{t}\in \Lambda$ be given.
 The horizontal derivative of $F$ at $X_{t}$ is defined as
 \begin{equation*}
   \Delta_{t}F(X_{t})=\lim_{h\rightarrow 0^{+}}\frac{F(X_{t,t+h})-F(X_{t})}{h}
 \end{equation*}
 if the corresponding limit exists.
 If $F$ is horizontally differentiable at each $X_{t}\in \Lambda$, we say $F$ is horizontally differentiable in $\Lambda$.
 \end{definition}

 \begin{definition}
 Let $F:\Lambda\rightarrow \mathbb{R}$ and $X_{t}\in \Lambda$ be given.
 The vertical derivative of $F$ at $X_{t}$ is defined as
 \begin{equation*}
   \Delta_{x}F(X_{t})=\lim_{h\rightarrow 0}\frac{F(X_{t}^{h})-F(X_{t})}{h}
 \end{equation*}
 if the corresponding limit exists.
 $F$ is said to be vertically differentiable in $\Lambda$ if the vertical derivative of $F$ at each $X_{t}\in \Lambda$ exists.
 We can similarly define $\Delta_{xx}F(X_{t})$.
 \end{definition}

 \begin{definition}
 Define $\mathbb{C}^{j,k}(\Lambda)$ as the set of functionals $F$ defined on $\Lambda$ which are $j$ orders horizontally
 and $k$ orders vertically differentiable in $\Lambda$ such that all these derivatives are $\Lambda$-continuous.
 \end{definition}

 \begin{example}
  If $F(X_{t})=f(t,X(t))$, with $f\in \mathbb{C}^{1,2}([0,T]\times \mathbb{R})$, then
  \begin{equation*}
     \Delta_{t}F= \partial _{t}f, \ \ \ \Delta_{x}F= \partial _{x}f, \ \ \ \Delta_{xx}F= \partial _{xx}f,
  \end{equation*}
  which are the classic derivatives.
  In general, these derivatives also satisfy the classic properties: Linearity, product and chain rule.
 \end{example}

 \begin{definition}
 For a functional $F$ from $\Lambda\rightarrow \mathbb{R}$, the functional It\^{o} type integral is defined as
 \begin{equation}\label{21}
   \int_0^T F(X_{t}) dX(t) := \lim_{n\rightarrow \infty} \sum_{i=1}^{n} F(X_{t_{i-1}})(X(t_{i})-X(t_{i-1}))
 \end{equation}
 if the limit exists, where $0=t_{0}< t_{1}<...<t_{n}=T$ is a partition of $[0,T]$.
 \end{definition}

 The following functional It\^{o} formula was firstly obtained by Dupire \cite{Dup} and then by Cont and Fourni\'{e} \cite{Con} for a more general formulation.
 \begin{theorem}\label{20}
 Let $(\Omega,\mathcal{F},(\mathcal{F}_{t})_{t\in[0,T]},P)$ be a probability space and $W$ be a classical Brownian motion.
 Let $X$ be an It\^{o} process of the form
 \begin{equation}\label{23}
   X(t)=X(0)+\int_0^t \psi(s) ds + \int_0^t \varphi(s) dW(s),
 \end{equation}
 where $X(0)$ is a constant, $\psi(t)$ and $\varphi(t)$ are progressively measurable processes and satisfying
 $ \mathbb{E} \int_0^T (|\psi(t)|^{2} + |\varphi(t)|^{2}) dt  < \infty.$
 If $F$ is in $\mathbb{C}^{1,2}(\Lambda)$, then for each $t\in[0,T]$,
  \begin{equation}\label{24}
   \begin{split}
    F(X_{t})=&\ F(X_{0})+\int_0^t \Delta_{s}F(X_{s}) ds+\int_0^t \Delta_{x}F(X_{s}) \psi(s) ds \\
           & + \int_0^t \Delta_{x}F(X_{s})\varphi(s) dW(s) +\frac{1}{2}\int_0^t \Delta_{xx}F(X_{s})\varphi(s)^{2} ds.
   \end{split}
  \end{equation}
 \end{theorem}

\subsection{Fractional Brownian motion}

Let $\Omega = C_{0}([0,T])$ be the space of continuous functions $\omega$ from $[0,T]$ to $\mathbb{R}$ with $\omega(0) = 0$, and
let $(\Omega,\mathcal{F},P)$ be a complete probability space.
The coordinate process $B^{H}:\Omega\rightarrow \mathbb{R}$ defined as
\begin{equation*}
  B^{H}(t,\omega)=\omega(t), \ \ \omega\in \Omega
\end{equation*}
is a Gaussian process satisfying (\ref{1}).
The process $B^{H}=\{ B^{H}(t),t\geq 0 \}$ is called the canonical fractional Brownian motion with Hurst parameter $H\in (0,1)$.
When $H=\frac{1}{2}$, this process is the classical Brownian motion denoted by $\{ W(t),t\geq 0 \}$.
It is elementary to verify that the fractional Brownian motions $B^{H}$  are not semi-martingales if $H\neq\frac{1}{2}$.
Throughout this paper we assume that $H\in(\frac{1}{2},1)$ is arbitrary but fixed.

Let $L^{2}(\Omega,\mathcal{F},P)$ be the space of all random variables $F:\Omega\rightarrow \mathbb{R}$ such that
$\mathbb{E}[|F|^{2}]< \infty.$
Denote $\phi(x) = H(2H - 1)|x|^{2H-2}, \ x \in \mathbb{R}$.
Let $\xi$ and $\eta$ be two continuous functions on $[0,T]$.
We define
\begin{equation*}
  \langle \xi,\eta \rangle_{t} = \int_0^t \int_0^t \phi(u-v) \xi(u) \eta(v) dudv,
\end{equation*}
and $ \| \xi \|_{t}^{2} =  \langle \xi,\xi  \rangle_{t}$.
Note that, for any $t\in [0,T], \ \langle \xi,\eta \rangle_{t}$ is a Hilbert scalar product.
Let $\mathcal{H}$ be the completion of the continuous functions under this Hilbert norm.
The elements of $\mathcal{H}$ may be distributions.
Let $|\mathcal{H}|$ be the linear space of measurable functions $\xi$ on $[0,T]$ such that
\begin{equation*}
  \| \xi \|_{|\mathcal{H}|}^{2} = \int_0^T \int_0^T \phi(u-v) |\xi(u)| |\xi(v)| dudv.
\end{equation*}
It is not difficult to show that $|\mathcal{H}|$ is a Banach space with the norm $\| \cdot \|_{|\mathcal{H}|}$.
Let $\mathcal{P}_{T}$ be the set of all polynomials of fractional Brownian motion in $[0,T]$, i.e., it contains all elements of the form
\begin{equation*}
  F(\omega) = f \left(\int_0^T \xi_{1}(t) dB^{H}(t),...,\int_0^T \xi_{n}(t) dB^{H}(t) \right),
\end{equation*}
where $n\geq 1, \ f\in \mathbb{C}^{\infty}_{b}(\mathbb{R}^{n})$ ($f$ and all its partial derivatives are bounded).
The Malliavin derivative operator $D_{s}^{H}$ of an element $F\in \mathcal{P}_{T}$ is defined as follows:
\begin{equation}\label{211}
  D_{s}^{H}F = \sum\limits_{i=1}^{n} \frac{\partial f}{\partial x_{i}}
               \left(\int_0^T \xi_{1}(t) dB^{H}(t),...,\int_0^T \xi_{n}(t) dB^{H}(t) \right)\xi_{i}(s), \ \ s\in [0,T].
\end{equation}
Since the divergence operator
  $D^{H}:L^{2}(\Omega,\mathcal{F}, P) \rightarrow (\Omega,\mathcal{F}, \mathcal{H})$ is closable,
we can consider the space $\mathbb{D}^{1,2}$ be the completion of $\mathcal{P}_{T}$ with the norm
\begin{equation*}
  \| F \|^{2}_{1,2} = \mathbb{E}|F|^{2} + \mathbb{E}\|D^{H}_{s} F\|^{2}_{T}.
\end{equation*}
In a similar way, given a Hilbert space $V$ we denote by  $\mathbb{D}^{1,2}(V)$ the corresponding Sobolev space of $V$-valued random variables.
We also introduce another derivative
\begin{equation*}
  D_{t}^{\phi}F = \int_0^T \phi(t-s) D_{s}^{H}F ds.
\end{equation*}

\begin{proposition}[Duncan et al. \cite{Hu}]\label{25}
Let $g\in\mathcal{H}$, $F$ and $\langle D^{H}F,g \rangle_{T}$ belong to $L^{2}(\Omega,\mathcal{F}, P)$, then
\begin{equation}\label{28}
  F \diamond \int_0^T g(t) dB^{H}(t) = F \int_0^T g(t) dB^{H}(t) -  \langle D^{H}F,g \rangle_{T},
\end{equation}
where ``$\diamond$'' denotes the Wick product.
\end{proposition}

\begin{proposition}[Duncan et al. \cite{Hu}]\label{26}
If $F(s)$ is a continuous stochastic process such that
 $ \mathbb{E} \big[ \| F \|_{T}^{2} + \big(\int_0^T D^{\phi}_{s}F(s) ds\big)^{2} \big] < \infty,$
then the Wick-It\^{o}-Skorohod type stochastic integral
\begin{equation}\label{29}
\int_0^T F(s) \diamond d B^{H}(s):= \lim_{n \rightarrow 0} \sum_{i=1}^{n} F(t_{i-1}) \diamond (B^{H}(t_{i})-B^{H}(t_{i-1}))
\end{equation}
exists in $L^{2}(\Omega,\mathcal{F}, P)$.
 Moreover, we have
$\mathbb{E} \int_0^T F(s) \diamond d B^{H}(s) = 0,$
and
\begin{equation*}
   \mathbb{E} | \int_0^T F(s) \diamond d B^{H}(s) |^{2}
 = \mathbb{E} \big[ \| F \|_{T}^{2} + \big(\int_0^T D^{\phi}_{s}F(s) ds\big)^{2} \big].
\end{equation*}
\end{proposition}

\section{Stratonovich type integral}

In this section, we study Stratonovich type integral with respect to c\`{a}dl\`{a}g functions.
\begin{definition}
 For a functional $F$ from $\Lambda\rightarrow \mathbb{R}$, the Stratonovich type integral is defined as
 \begin{equation}\label{41}
   \int_0^T F(X_{t}) \circ dX(t) := \lim_{n\rightarrow \infty} \sum_{i=1}^{n} F(X_{\frac{t_{i-1}+t_{i}}{2}})(X(t_{i})-X(t_{i-1}))
 \end{equation}
 whenever the limit exists in the sense of $L^{2}(\Omega,\mathcal{F},P)$, where $0=t_{0}< t_{1}<...<t_{n}=T$ is a partition of $[0,T]$.
It is actually a classical Stratonovich type integral as the integrand is adapted.
 \end{definition}

Next we discuss  the Stratonovich type integral driven by classical Brownian motion and fractional Brownian motion respectively.

\subsection{The case of classical Brownian motion}\label{30}

We firstly study the relation between Stratonovich integral (\ref{41}) and It\^{o} integral (\ref{21}).
It points out that there is a little difference for the relation between It\^{o} integral and Stratonovich integral comparing with the classical case.

Suppose $F:\Lambda\rightarrow \mathbb{R}$ is both It\^{o} integrable and Stratonovich integrable with respect to classical Brownian motion.
Let $\pi:0=t_{0}<t_{1}<t_{2}<\cdot \cdot \cdot <t_{n}=T$ be a partition of the interval $[0,T]$.
From the definition, we have
\begin{align}\label{33}
    \int_0^T F(W_{t}) \circ dW(t)
 = &\lim_{n\rightarrow \infty} \sum_{i=1}^{n} F(W_{\frac{t_{i-1}+t_{i}}{2}})(W(t_{i})-W(t_{i-1})) \nonumber\\
 = &\lim_{n\rightarrow \infty} \bigg[ \sum_{i=1}^{n} F(W_{\frac{t_{i-1}+t_{i}}{2}})(W(t_{i})-W(\frac{t_{i-1}+t_{i}}{2})) \nonumber\\
   &+ \sum_{i=1}^{n} \big(F(W_{\frac{t_{i-1}+t_{i}}{2}}) - F(W_{t_{i-1}})\big)\big(W(\frac{t_{i-1}+t_{i}}{2}) - W(t_{i-1})\big) \nonumber\\
   &+ \sum_{i=1}^{n} F(W_{t_{i-1}})\big(W(\frac{t_{i-1}+t_{i}}{2}) - W(t_{i-1})\big)\bigg] \nonumber\\
 =:& \lim_{n\rightarrow \infty} (A_{1}^{n} + A_{2}^{n} + A_{3}^{n}).
\end{align}
It is easy to see that
\begin{equation*}
   \lim_{n\rightarrow \infty}(A_{1}^{n} + A_{3}^{n}) = \int_0^T F(W_{t}) dW(t).
\end{equation*}
For the second term $A_{2}^{n}$ in the right hand side of (\ref{33}), for simplicity, we denote
$$t_{i-1}^{i}=\frac{t_{i-1}+t_{i}}{2}, \ \ \  \delta t_{i-1}^{i}=t_{i-1}^{i}-t_{i-1}, \ \ \ \delta W_{i-1}^{i}=W(t_{i-1}^{i}) - W(t_{i-1}).$$
Define
\begin{equation*}
\begin{split}
  Y_{t_{i-1}^{i}}(s)=&\ W(s)1_{[0,t_{i-1})}(s) + W(t_{i-1})1_{[t_{i-1},t_{i-1}^{i})}(s) + W(t_{i-1}^{i})1_{\{t_{i-1}^{i}\}},\\
  Z_{t_{i-1}^{i}}(s)=&\ W(s)1_{[0,t_{i-1})}(s) + W(t_{i-1})1_{[t_{i-1},t_{i-1}^{i}]}(s).
\end{split}
\end{equation*}
If $F\in \mathbb{C}^{1,1}(\Lambda)$, then
\begin{gather}\label{42}
\begin{align}
 A_{2}^{n}&= \sum_{i=1}^{n}\big(F(W_{t_{i-1}^{i}}) - F(W_{t_{i-1}})\big) \delta W_{i-1}^{i} \nonumber\\
      &= \sum_{i=1}^{n}\bigg[F(W_{t_{i-1}^{i}})- F(Y_{t_{i-1}^{i}}) + F(Y_{t_{i-1}^{i}}) - F(Z_{t_{i-1}^{i}})
       + F(Z_{t_{i-1}^{i}}) - F(Y_{t_{i-1}})\bigg]\delta W_{i-1}^{i} \nonumber\\
      &= \sum_{i=1}^{n}\bigg[\big(F(W_{t_{i-1}^{i}})- F(Y_{t_{i-1}^{i}})\big) + \Delta_{x} F(Z^{h_{i}}_{t_{i-1}^{i}}) \delta W_{i-1}^{i} +\Delta_{t} F(Y_{t_{i-1},y^{i}}) \delta t_{i-1}^{i}\bigg] \delta W_{i-1}^{i} \nonumber\\
      &=:A_{2,1}^{n}+A_{2,2}^{n}+A_{2,3}^{n},
\end{align}
\end{gather}
where $h_{i}\in (0,\delta W_{i-1}^{i})$ and $y_{i}\in (0,\delta t_{i-1}^{i})$.
By $\Lambda$-continuity of $F$, the term
   $A_{2,1}^{n}\xrightarrow{n} 0$.
Similarly, by ordinary dominated convergence,
\begin{equation*}
  A_{2,2}^{n}=\sum_{i=1}^{n} \Delta_{x} F(Z^{h_{i}}_{t_{i-1}^{i}}) (\delta W_{i-1}^{i})^{2} \stackrel{n}{\longrightarrow}
    \frac{1}{2}\int_0^T \Delta_{x} F(W_{t}) dt.
\end{equation*}
For the term $A_{2,3}^{n}$, it converges to 0 in the sense of mean square.
In fact,
\begin{equation*}
    \mathbb{E} \bigg[\big(\sum_{i=1}^{n} \Delta_{t}  F(Y_{t_{i-1},y^{i}}) \delta t_{i-1}^{i} \delta W_{i-1}^{i} \big)^{2}\bigg]
  = \sum_{i=1}^{n} \mathbb{E}\big(\Delta_{t}  F(Y_{t_{i-1},y^{i}}) \big)^{2} (\delta t_{i-1}^{i})^{3}
  \stackrel{n}{\longrightarrow} 0.
\end{equation*}
We conclude from above:
\begin{equation*}
    \int_0^T F(W_{t}) \circ dW(t) = \int_0^T F(W_{t}) dW(t) +  \frac{1}{2}\int_0^T \Delta_{x} F(W_{t}) dt.
\end{equation*}

\begin{remark}
It is interesting for the emergence of the term $A^{n}_{2,3}$ in (\ref{42}), which is non-existent in the classical case.
Although the term $A^{n}_{2,3}$ is zero as $n\rightarrow \infty$, we need $F\in \mathbb{C}^{1,1}(\Lambda)$,
not in $\mathbb{C}^{0,1}(\Lambda)$ of the classical case.
\end{remark}

From the above discussion, we directly have the following result.
\begin{proposition}\label{43}
Let $F$ be a functional belongs to $\mathbb{C}^{1,1}(\Lambda)$, if $\{X(t), t\geq 0\}$ is an It\^{o} process defined in (\ref{23}), then
\begin{equation*}
    \int_0^T F(X_{t}) \circ dW(t) = \int_0^T F(X_{t}) dW(t) +  \frac{1}{2}\int_0^T \Delta_{x} F(X_{t}) \varphi(t) dt.
\end{equation*}
\end{proposition}

As a consequence of the above proposition, recalling Theorem \ref{20}, it is easy to obtain the following functional It\^{o}-Stratonovich formula.

\begin{theorem}
Suppose $F\in \mathbb{C}^{1,2}(\Lambda)$, and $X$ is a stochastic process of the form
 \begin{equation*}
   X(t)=X(0)+\int_0^t \psi(s) ds + \int_0^t \varphi(s) \circ dW(s),
 \end{equation*}
where $X(0)$ is a constant, $\psi$ and $\varphi$ are two progressively
measurable processes satisfying
$$ \mathbb{E} \left[\int_0^T |\psi(t)|^{2} dt + \int_0^T |\varphi(t)|^{2} dt \right] < \infty,$$
then
  \begin{equation*}\label{}
   \begin{split}
    F(X_{t})=& \ F(X_{0})+\int_0^t \Delta_{s}F(X_{s}) ds
     +\int_0^t \Delta_{x}F(X_{s}) \psi(s) ds + \int_0^t \Delta_{x}F(X_{s})\varphi(s)\circ dW(s)\\
            =& \ F(X_{0})+\int_0^t \Delta_{s}F(X_{s}) ds+\int_0^t \Delta_{x}F(X_{s})\circ dX(s).
   \end{split}
  \end{equation*}
\end{theorem}

\begin{remark}
One similar (but different) functional It\^{o}-Stratonovich formula for classical Bromtion motion is presented in  Buckdahn et al. \cite{Buckdahn},
where they consider the ``path-derivatives''.
\end{remark}

\subsection{The case of fractional Brownian motion}

In this subsection, we study the Stratonovich type integral with respect to the fractional Brownian motion.
The classical Stratonovich integral driven by fractional Brownian motion has been studied by Al\`{o}s and Nualart \cite{Nualart1},
where it is equivalent to that of Russo-Vallois integral \cite{Russo}.

Similar to the discussion of Subsection \ref{30},
it's easy to obtain the following relation between Stratonovich integral and It\^{o} integral
when they are both driven by the fBm.

\begin{proposition}\label{45}
Let $F\in \mathbb{C}^{1,1}(\Lambda)$ and $\{X(t), t\geq 0\}$ be a stochastic process of the form (\ref{31}) below,
then
\begin{equation*}
    \int_0^T F(X_{t}) \circ dB^{H}(t) = \int_0^T F(X_{t}) dB^{H}(t).
\end{equation*}
\end{proposition}

\begin{remark}\label{35}
The above proposition implies that for the fractional Brownian motion,
the Stratonovich type integral (\ref{41}) is equivalent to the It\^{o} type integral  (\ref{21}).
In fact, similar to the classical case, the Stratonovich type integral (\ref{41}) with respect to the fractional Brownian motion
is also equivalent to the Russo-Vallois integral.
\end{remark}

In the following, when consider the fractional Brownian motion,
we don't distinguish Stratonovich type integral and It\^{o} type integral since they are equivalent,
and we continue to use the notation of Stratonovich integral in order to keep uniformity.
Next, we give a functional It\^{o} formula for the fractional Brownian motion.

Let $\{X(t),0\leq t\leq T\}$ be a stochastic process of the form
\begin{equation}\label{31}
     X(t)=X(0)+\int_0^t \psi(s) ds + \int_0^t \varphi(s) \circ dB^{H}(s),
\end{equation}
where $X(0)$ is a constant, $\psi(t,\omega)$ and $\varphi(t,\omega):[0,T]\times \Omega \rightarrow \mathbb{R}$
are two progressively measurable processes with $\varphi$ is Stratonovich integrable, such that
\begin{equation*}
  \mathbb{E} \left[\int_0^T |\psi(t)|^{2} dt + \int_0^T |\varphi(t)|^{2} dt \right] < \infty.
\end{equation*}

\begin{theorem}\label{32}
Let $\{X(t),0\leq t\leq T\}$ be defined in (\ref{31}). If $F\in \mathbb{C}^{1,2}(\Lambda)$, then for any $T\geq 0$,
  \begin{equation}\label{34}
    F(X_{T})=F(X_{0})+\int_0^T \Delta_{t}F(X_{t}) dt
     + \int_0^T \Delta_{x}F(X_{t}) \psi(t) dt + \int_0^T \Delta_{x}F(X_{t})\varphi(t) \circ  dB^{H}(t).
  \end{equation}
\end{theorem}


\begin{proof}
We take a sequence of nested subdivisions of $[0,T], \ \pi: 0=t_{0}< t_{1}<...<t_{n}=T$.
Define
\begin{equation*}
  Y_{t_{i}}(s)=\sum_{j=1}^{i}X(t_{j-1})I_{[t_{j-1},t_{j})}(s) + X(t_{i})I_{\{t_{i}\}}, \ \
  Z_{t_{i}}(s)= \sum_{j=1}^{i}X(t_{j-1})I_{[t_{j-1},t_{j})}(s)+ X(t_{i-1})I_{\{t_{i}\}}.
\end{equation*}
Moreover, we denote
$$\delta t_{i}=t_{i}-t_{i-1}, \ \ \ \delta B^{H}_{i}=B^{H}(t_{i}) - B^{H}(t_{i-1}), \ \ \ \delta X_{i}=X(t_{i})-X(t_{i-1}).$$
Since $\psi(t)$ and $\varphi(t)$ are elementary functions,
we may wish to set $\psi(t)=\sum\limits_{i}\psi(t_{i-1})I_{[t_{i-1},t_{i})}$ and $\varphi(t)=\sum\limits_{i}\varphi(t_{i-1})I_{[t_{i-1},t_{i})}$, then
\begin{equation*}
  \delta X_{i} = \int_{t_{i-1}}^{t_{i}} \psi(s) ds + \int_{t_{i-1}}^{t_{i}} \varphi(s) dB^{H}(s)
               = \psi(t_{i-1})\delta t_{i} + \varphi(t_{i-1}) \delta B^{H}_{i}.
\end{equation*}
Notice that $Y_{0}=Z_{0}=X_{0}$, then
\begin{equation*}
  F(X_{T}) - F(X_{0}) = F(X_{T}) - F(Y_{t_{n}}) + \sum_{i=1}^{n} (F(Y_{t_{i}}) - F(Y_{t_{i-1}})).
\end{equation*}
Since $F\in \mathbb{C}^{1,2}(\Lambda)$, from the Taylor's theorem,
we get the existence of $h_{i}\in (0,\delta X_{i})$ and $y_{i}\in (0,\delta t_{i})$ such that,
\begin{equation*}
  \begin{split}
         F(Y_{t_{i}}) - F(Y_{t_{i-1}})
      =& \ F(Y_{t_{i}}) - F(Z_{t_{i}}) + F(Z_{t_{i}}) - F(Y_{t_{i-1}})\\
      =& \ \Delta_{x}F(Z_{t_{i}})\delta X_{i} +
       \frac{1}{2}\Delta_{xx}F(Z_{t_{i}}^{h_{i}})(\delta X_{i})^{2} + \Delta_{t}F(Y_{t_{i-1},y_{i}})\delta t_{i}.
  \end{split}
\end{equation*}
Hence, we get
\begin{equation*}
  F(X_{T}) - F(X_{0}) = A_{1}^{n} + A_{2}^{n} + A_{3}^{n} + \frac{1}{2}A_{4}^{n} + A_{5}^{n},
\end{equation*}
where
\[\begin{split}
  A_{1}^{n}& = F(X_{T}) - F(Y_{t_{n}}), \\
  A_{2}^{n}&  = \sum_{i=1}^{n} \Delta_{x}F(Z_{t_{i}})\psi(t_{i-1})\delta t_{i}, \\
  A_{3}^{n}& = \sum_{i=1}^{n} \Delta_{x}F(Z_{t_{i}})\varphi(t_{i-1}) \delta B^{H}_{i}, \\
  A_{4}^{n}&  = \sum_{i=1}^{n} \Delta_{xx}F(Z_{t_{i}}^{h_{i}})(\delta X_{i})^{2}, \\
  A_{5}^{n}& = \sum_{i=1}^{n}  \Delta_{t}F(Y_{t_{i-1},y_{i}})\delta t_{i}.
\end{split}\]
For term $A_{1}^{n}$: Since $X(t)$ is continuous on interval $[0,T]$, hence uniformly continuous,
and $Y_{t_{n}}$ converges uniformly (in the $\Lambda$-distance) to $X_{T}$.
By $\Lambda$-continuity of $F$, we have $A_{1}^{n}\xrightarrow[]{n} 0$.

Similarly, for terms $A_{2}^{n}, \ A_{3}^{n}$ and $A_{5}^{n}$: When $t\in[t_{i-1},t_{i})$, it is easy to know
\begin{equation*}
\Delta_{x}F(Z_{t_{i}}) \stackrel{n}{\longrightarrow} \Delta_{x}F(X_{t}); \ \ \
\Delta_{t}F(Y_{t_{i-1},y_{i}})\stackrel{n}{\longrightarrow} \Delta_{t}F(X_{t}).
\end{equation*}
Hence, by ordinary dominated convergence,
\begin{equation*}
\begin{split}
   A_{2}^{n} &= \sum_{i=1}^{n} \Delta_{x}F(Z_{t_{i}})\psi(t_{i-1})\delta t_{i} \stackrel{n}{\longrightarrow} \int_0^T \Delta_{x}F(X_{t}) \psi(t) dt,\\
   A_{3}^{n} &= \sum_{i=1}^{n} \Delta_{x}F(Z_{t_{i}})\varphi(t_{i-1}) \delta B^{H}_{i} \stackrel{n}{\longrightarrow}
    \int_0^T \Delta_{x}F(X_{t})\varphi(t) dB^{H}(t),\\
   A_{5}^{n} &= \sum_{i=1}^{n} \Delta_{t}F(Y_{t_{i-1},y_{i}})\delta t_{i} \stackrel{n}{\longrightarrow} \int_0^T \Delta_{t}F(X_{t}) dt.
\end{split}
\end{equation*}
Next, we consider the term $A_{4}^{n}$,
\begin{equation*}
\begin{split}
       A_{4}^{n}
     =&  \sum_{i=1}^{n} \Delta_{xx}F(Z_{t_{i}}^{h_{i}})\psi(t_{i-1})^{2}(\delta t_{i})^{2}
       +2\sum_{i=1}^{n} \Delta_{xx}F(Z_{t_{i}}^{h_{i}})\psi(t_{i-1})\varphi(t_{i-1})\delta t_{i} \delta B^{H}_{i}\\
      &+ \sum_{i=1}^{n} \Delta_{xx}F(Z_{t_{i}}^{h_{i}})\varphi(t_{i-1})^{2} (\delta B^{H}_{i})^{2}\\
     =&: A_{4,1}^{n}+A_{4,2}^{n}+A_{4,3}^{n} .
\end{split}
\end{equation*}
It's easy to see that $ A_{4,1}^{n}\xrightarrow{n} 0$.
Since $\frac{1}{2}< H<1$, the quadratic variation of the fractional Brownian motion is zero. Hence
$A_{4,2}^{n}$ and $A_{4,3}^{n}$ also converge to 0 as $n\rightarrow \infty$.
Therefore, $ A_{4}^{n}\xrightarrow{n} 0$.
Our desired result is proved.
\end{proof}

\begin{remark}
The above proof is probabilistic and makes use of the proof of classical It\^{o} formula.
A nonprobabilistic proof of a general result, i.e., a functional It\^{o} formula with respect to a process with finite quadratic variation,
was established in Cont and Fourni\^{e} \cite{Cont10}.
Cosso and Russo \cite{Russo4} also obtained a functional It\^{o} formula with respect to a process with finite quadratic variation via regularization approach.
\end{remark}

\section{Wick-It\^{o}-Skorohod type integral}

In this section, we study the Wick-It\^{o}-Skorohod type integral with respect to fractional Brownian motion via Malliavin calculus approach.
It's noted that the Stratonovich type integral $\int_0^t F(X_{t}) \circ dB^{H}(t)$ does $\mathbf{not}$ satisfy the following property:
$$\mathbb{E}\int_0^t F(X_{t})\circ  dB^{H}(t)=0.$$
In this section, we study a new type of stochastic integral $\int_0^t F(X_{t}) \diamond dB^{H}(t)$ satisfying
$$\mathbb{E}\int_0^t F(X_{t}) \diamond dB^{H}(t)=0.$$
In particular, we consider the simple case, i.e., the integrand $F=F(B^{H}_{t})$.
For general case of the integrand, some further studies will be given in the coming future researches.


\begin{definition}
Let $F:\Lambda \rightarrow \mathbb{R}$ be $\Lambda$-continuous  and vertical differentiable in $\Lambda$.
For any fixed $t\in [0,T]$, the Malliavin derivative of $F=F(B^{H}_{t})$ is defined as:
\begin{equation}\label{51}
  D^{H}_{s}F= \Delta_{x} F(B^{H}_{t}) I_{[0,t]}(s), \ \ 0\leq s\leq T.
\end{equation}
\end{definition}


By using a similar method, we can check that the special Malliavin derivative $D^{H}F$ defined in (\ref{51}) satisfies (\ref{28}),
and the Wick-It\^{o}-Skorohod type stochastic integral, similar to (\ref{29}), is defined as
\begin{equation}\label{52}
\int_0^T F(B^{H}_{t}) \diamond dB^{H}(t):= \lim_{n \rightarrow 0} \sum_{i=1}^{n} F(B^{H}_{t_{i-1}}) \diamond (B^{H}(t_{i})-B^{H}(t_{i-1}))
\end{equation}
in the sense of $L^{2}(\Omega,\mathcal{F}, P)$,
where $0=t_{0}\leq t_{1}\leq ... \leq t_{n}=T$ is the partition of the interval $[0.T]$.
It satisfies the property $\mathbb{E} \big(\int_0^T F(B^{H}_{t}) \diamond dB^{H}(t) \big) = 0$.
In fact,
\begin{equation*}
\begin{split}
  &\mathbb{E} \left( \sum_{i=1}^{n} F(B^{H}_{t_{i-1}}) \diamond (B^{H}(t_{i})-B^{H}(t_{i-1})) \right)\\
 =&\sum_{i=1}^{n} \mathbb{E} \bigg[ F(B^{H}_{t_{i-1}}) \diamond \big(B^{H}(t_{i})-B^{H}(t_{i-1})\big) \bigg]\\
 =&\sum_{i=1}^{n} \mathbb{E} \big[ F(B^{H}_{t_{i-1}})\big] \mathbb{E} \big[ B^{H}(t_{i})-B^{H}(t_{i-1}) \big]
 =0.
\end{split}
\end{equation*}
For the integral (\ref{52}), it owns a similar result to  Proposition \ref{26}.
Since the proof is identical to the proof of  Proposition \ref{26}, we omit the details and only state the main result for simplicity of presentation.

\begin{theorem}
If $F=F(B^{H}_{t})$ is $\Lambda$-continuous  such that
\begin{equation}\label{57}
  \mathbb{E} \big[\big(\int_0^T D^{\phi}_{s}F(B^{H}_{s}) ds\big)^{2} + \int_0^T \int_0^T \phi(u-v) F(B^{H}_{u}) F(B^{H}_{v}) dudv \big] < \infty,
\end{equation}
then the integral (\ref{52}) exists in $L^{2}(\Omega,\mathcal{F}, P)$.
Moreover,
$$\mathbb{E} \int_0^T F(B^{H}_{s}) \diamond d B^{H}(s) = 0,$$
and
\begin{equation*}
\mathbb{E} | \int_0^T F(B^{H}_{s}) \diamond d B^{H}(s) |^{2}
  = \mathbb{E} \big[\big(\int_0^T D^{\phi}_{s}F(B^{H}_{s}) ds\big)^{2}
   + \int_0^T \int_0^T \phi(u-v) F(B^{H}_{u}) F(B^{H}_{v}) dudv \big].
\end{equation*}
\end{theorem}

The following proposition is a relation between Wick-It\^{o}-Skorohod integral and Stratonovich integral.
It was an extension of Proposition 3.3 in  Nualart \cite{Nualart} for the classical result.

\begin{proposition}\label{54}
Let $F:\Lambda \rightarrow \mathbb{R}$ be in $\mathbb{C}^{1,1}(\Lambda)$ such that $F=F(B^{H}_{t})$ satisfies (\ref{57}),
then
\begin{equation}\label{53}
\int_0^T F(B^{H}_{t}) \diamond dB^{H}(t) = \int_0^T F(B^{H}_{t}) \circ dB^{H}(t) - H \int_0^T \Delta_{x}F(B^{H}_{t})t^{2H-1} dt.
\end{equation}
\end{proposition}

\begin{proof}
Let $\pi:0=t_{0}\leq t_{1}\leq ... \leq t_{n}=T$ be a partition of the interval $[0.T]$.
The formula (\ref{28}) yields that
\begin{equation*}
\begin{split}
   &\sum_{i=1}^{n} F(B^{H}_{t_{i-1}}) \diamond (B^{H}(t_{i})-B^{H}(t_{i-1}))\\
  =&\sum_{i=1}^{n} F(B^{H}_{t_{i-1}}) (B^{H}(t_{i})-B^{H}(t_{i-1}))
   -\sum_{i=1}^{n} \langle D^{H}_{s}  F(B^{H}_{t_{i-1}}),I_{[t_{i-1},t_{i}]} \rangle_{T},\\
\end{split}
\end{equation*}
where $D^{H}_{s}  F(B^{H}_{t_{i-1}})= \Delta_{x}F(B^{H}_{t_{i-1}}) I_{[0,t_{i-1}]}(s)$ and notice Proposition \ref{45}. We have
\begin{equation*}
\begin{split}
   &\Delta_{x}F(B^{H}_{t_{i-1}}) \langle I_{[0,t_{i-1}]},I_{[t_{i-1},t_{i}]} \rangle_{T}\\
  =&\Delta_{x}F(B^{H}_{t_{i-1}})\big(\langle I_{[0,t_{i-1}]},I_{[0,t_{i}]} \rangle_{T} - \langle I_{[0,t_{i-1}]},I_{[0,t_{i-1}]} \rangle_{T}\big)\\
  =&\frac{1}{2} \Delta_{x}F(B^{H}_{t_{i-1}})\big[ t_{i}^{2H} - t_{i-1}^{2H} - (t_{i} - t_{i-1})^{2H} \big]\\
  =&\frac{1}{2} \Delta_{x}F(B^{H}_{t_{i-1}})\big[ 2H \hat{t}_{i-1}^{2H-1}(t_{i} - t_{i-1}) - (t_{i} - t_{i-1})^{2H}\big],
\end{split}
\end{equation*}
where $\hat{t}_{i-1}\in (t_{i-1},t_{i})$.
It is easy to know that $(t_{i} - t_{i-1})^{2H}=o(t_{i} - t_{i-1})\stackrel{n}{\rightarrow} 0$.
Then
\begin{equation*}
   \sum_{i=1}^{n} \langle D^{H}_{s}  F(B^{H}_{t_{i-1}}),I_{[t_{i-1},t_{i}]} \rangle_{T}
   \stackrel{n}{\longrightarrow}  H \int_0^T \Delta_{x}F(B^{H}_{t})t^{2H-1} dt.
\end{equation*}
 This completes the proof.
\end{proof}

\begin{remark}
 Formula (\ref{53}) leads to the following equation for the expectation of the integral (\ref{21}) with respect to fBm:
\begin{equation*}
    \mathbb{E} \int_0^T F(B^{H}_{t}) \circ dB^{H}(t) = H \int_0^T \mathbb{E}(\Delta_{x}F(B^{H}_{t}))t^{2H-1} dt.
\end{equation*}

\end{remark}

From Theorem \ref{32} and Proposition \ref{54}, we directly obtain the following functional It\^{o} formula for Wick-It\^{o}-Skorohod integral.

\begin{theorem}\label{50}
 Suppose $F\in \mathbb{C}^{1,2}(\Lambda)$ such that $F(B^{H}_{t})$ satisfies (\ref{57}), then
\begin{equation*}
   F(B^{H}_{T}) = F(B^{H}_{0}) + \int_0^T \Delta_{t}F(B^{H}_{t}) dt
                  + \int_0^T \Delta_{x}F(B^{H}_{t})  \diamond dB^{H}(t) + H \int_0^T \Delta_{xx}F(B^{H}_{t})t^{2H-1} dt.
\end{equation*}
\end{theorem}

\section{Fractional BSDE}

As an application, in this section we study the fractional BSDEs with path-dependent coefficients.
We solve this class of fractional BSDEs by using a type of semilinear parabolic path-dependent PDEs.
The approach is based on a relationship between fractional BSDEs and semilinear PDEs.
For the recent developments of semilinear parabolic path-dependent PDEs we refer the readers to Ekren et al. \cite{Iekr,Iekr2,Iekr3}.

Denote $\mathcal{V}_{T} = \{ Y(t)=\phi(\gamma_{t})| \phi \in \mathbb{C}^{1,2}(\Lambda), \ \forall t\in [0,T] \}$,
and let $\mathcal{\widetilde{V}}_{T}$ be the completion of $\mathcal{V}_{T}$ under the following $\beta$-norm:
\begin{equation*}
  \| Y \|_{\beta}^{2} = \mathbb{E} \int_0^{T} e^{\beta t} |Y(t)|^{2} dt=\mathbb{E} \int_0^{T} e^{\beta t} |\phi(\gamma_{t})|^{2} dt.
\end{equation*}
Consider the following fractional BSDE with path-dependent coefficient:
\begin{equation}\label{61}
 \begin{cases}
  dY(t)= -f(B^{H}_{t},Y(t),Z(t))dt - Z(t) \diamond dB^{H}(t), \ \ 0\leq t\leq T, \\
   Y(T)= g(B^{H}_{T}).
 \end{cases}
\end{equation}
A pair of $\mathcal{F}_{t}$-adapted stochastic processes $\{(Y(t),Z(t));0\leq t\leq T\}$ is called a solution to the above equation if
\begin{equation*}
  Y(t)=g(B^{H}_{T}) + \int_t^T f(B^{H}_{s},Y(s),Z(s))ds + \int_t^T Z(s) \diamond dB^{H}(s),  \ \ 0\leq t\leq T.
\end{equation*}
We want to show that a solution in $\mathcal{\widetilde{V}}_{T}$  to the above fractional BSDE  exists uniquely.

In (\ref{61}), for the case of $f(B^{H}_{t},Y(t),Z(t))=f(B^{H}(t),Y(t),Z(t))$ and $g(B^{H}_{T})=g(B^{H}(T))$,
the existence and uniqueness theorem has been obtained by Hu and Peng \cite{Peng}.
If we consider classical Brownian motion instead of the fractional Brownian motion in (\ref{61}), it was also systemic studied by Peng and Wang \cite{Spen}.

Consider the following semilinear parabolic path-dependent PDE:
\begin{equation}\label{62}
 \begin{cases}
  \Delta_{t} u(\gamma_{t}) + \sigma (t)\Delta_{xx} u(\gamma_{t}) + f(\gamma_{t}, u(\gamma_{t}),-\Delta_{x} u(\gamma_{t}))=0,
   \ \  \gamma_{t}\in \Lambda_{t}, \ t\in[0,T), \\
  u(\gamma) = g(\gamma), \ \ \gamma\in \Lambda_{T},
 \end{cases}
\end{equation}
where $\sigma (t)=H t^{2H-1}$.
By applying Theorem \ref{50} to $u(B^{H}_{t})$, we have
\begin{equation*}
  \begin{split}
   du(B^{H}_{t}) =& \big[\Delta_{t} u(B^{H}_{t}) + \sigma (t)\Delta_{xx} u(B^{H}_{t})\big] dt + \Delta_{x} u(B^{H}_{t})  \diamond dB^{H}(t)\\
                 =& - f(B^{H}_{t}, u(B^{H}_{t}),-\Delta_{x} u(B^{H}_{t})) dt + \Delta_{x} u(B^{H}_{t})  \diamond dB^{H}(t),
   \end{split}
\end{equation*}
Thus we obtain the following theorem.
\begin{theorem}
If PDE (\ref{62}) has a solution $u$ which belongs to $\mathbb{C}^{1,2}(\Lambda)$,
then $(Y(t),Z(t)):=(u(B^{H}_{t}),-\Delta_{x} u(B^{H}_{t}))$ is a solution of the fractional BSDE (\ref{61}).
\end{theorem}

For a pair of solutions of Eq. (\ref{61}), we derive the following relation.

\begin{proposition}
Let BSDE (\ref{61}) has a solution of the form $(Y(t)=u(B^{H}_{t}),Z(t)= v(B^{H}_{t}))$,
where  $u\in \mathbb{C}^{1,2}(\Lambda)$. Then $-\Delta_{x} u(B^{H}_{t})=v(B^{H}_{t})$.
\end{proposition}

\begin{proof}
By the functional It\^{o} formula we have
\begin{equation*}
   du(B^{H}_{t}) = \big[\Delta_{t} u(B^{H}_{t}) + \sigma (t)\Delta_{xx} u(B^{H}_{t})\big] dt + \Delta_{x} u(B^{H}_{t})  \diamond dB^{H}(t).
\end{equation*}
Or we can rewrite as
\begin{equation*}
   u(B^{H}_{t}) =g(B^{H}_{T}) - \int_t^T \big[\Delta_{s} u(B^{H}_{s}) + \sigma (s)\Delta_{xx} u(B^{H}_{s})\big] ds
                 - \int_t^T \Delta_{x} u(B^{H}_{s})  \diamond dB^{H}(s).
\end{equation*}
Hence
\begin{equation*}
  \begin{split}
    & - \int_t^T \big[\Delta_{s} u(B^{H}_{s}) + \sigma (s)\Delta_{xx} u(B^{H}_{s})\big] ds - \int_t^T \Delta_{x} u(B^{H}_{s})  \diamond dB^{H}(s)\\
   =& \int_t^T f(B^{H}_{s},u(B^{H}_{s}),v(B^{H}_{s}))ds - \int_t^T v(B^{H}_{s}) \diamond dB^{H}(s).
  \end{split}
\end{equation*}
This is also true for $t=0$. Namely,
\begin{equation*}
  \begin{split}
    & - \int_0^T \big[\Delta_{s} u(B^{H}_{s}) + \sigma (s)\Delta_{xx} u(B^{H}_{s})\big] ds - \int_0^T \Delta_{x} u(B^{H}_{s})  \diamond dB^{H}(s)\\
   =& \int_0^T f(B^{H}_{s},u(B^{H}_{s}),v(B^{H}_{s}))ds - \int_0^T v(B^{H}_{s}) \diamond dB^{H}(s).
  \end{split}
\end{equation*}
Subtracting the above two equations, we deduce
\begin{equation*}
  \begin{split}
      &\int_0^t \big[\Delta_{s} u(B^{H}_{s}) + \sigma (s)\Delta_{xx} u(B^{H}_{s})+f(B^{H}_{s},u(B^{H}_{s}),v(B^{H}_{s}))\big] ds\\
        & + \int_0^t \big[\Delta_{x} u(B^{H}_{s})+v(B^{H}_{s})\big]  \diamond dB^{H}(s)=0,
  \end{split}
\end{equation*}
for all $t\in [0,T]$.
Then from Lemma 3.2 of Hu et al. \cite{Song}, we obtain
\begin{equation*}
  v(B^{H}_{t})=-\Delta_{x} u(B^{H}_{t}), \ \ \forall t\in (0,T).
\end{equation*}
This completes the proof.
\end{proof}

\begin{remark}
From the above proof, we also see that if the semilinear PDE (\ref{62}) has a unique solution, then BSDE (\ref{61}) also has a unique solution.
\end{remark}

Similar to Hu and Peng \cite{Peng} and Maticiuc and Nie \cite{Maticiuc}, we can also use the Picard iteration approach
to prove the existence and uniqueness of solutions of BSDE (\ref{61}).
Here we just present the result without the details proof.

\begin{theorem}
Let $f(\gamma, y, z)$ be uniformly  Lipschitz continuous with respect to $\gamma, y$ and $z$.
Let $g$ be $\Lambda$-continuously differentiable with bounded derivatives and of polynomial growth.
Then the fractional BSDE (\ref{61}) has a unique solution in $\mathcal{\widetilde{V}}_{T}$.
\end{theorem}

\section{Conclusions}

In this paper, we developed a functional It\^{o} calculus for fractional Brownian motion with Hurst parameter $H> \frac{1}{2}$.
In particular, the Stratonovich type and Wick-It\^{o}-Skorohod type integrals have been studied respectively.
The main result is the functional It\^{o} formulas for fractional Brownian motion.
As an application, we dealt with the fractional BSDEs with path-dependent coefficients.
A relation between this type of fractional BSDEs and path-dependent PDEs was also established.
In the coming future researches, we would devote to develop the application of the functional It\^{o} formulas that we established in this paper.
The functional It\^{o} calculus for fractional Brownian motion with Hurst parameter $H< \frac{1}{2}$ is also another goal.

\section*{Acknowledgements}

The authors would like to thank Professor Rama Cont for his helpful comments and discussions.

\end{document}